\documentclass[11pt,a4paper]{article}
\usepackage[utf8]{inputenc}
\usepackage{amsmath,amssymb,amsthm,amsfonts,latexsym,graphicx,subfigure,enumerate,tikz}
\usepackage[left=3.6cm, right=3.6cm, top=3.6cm, bottom=3.6cm]{geometry}
\usepackage{xcolor}

\theoremstyle{plain}
\newtheorem{theo}{Theorem}
\newtheorem{theorem}{Theorem}
\newtheorem{lemma}[theo]{Lemma}
\newtheorem{prop}[theo]{Proposition}

\theoremstyle{definition}

\newtheorem{definition}{Definition}
\newtheorem{remark}{Remark}
\newtheorem{ques}{Question}

\usepackage{hyperref}
\usepackage{comment}

\usepackage[affil-it]{authblk}

\title{Finite projective planes meet spectral gaps}
\author[1,2,3]{Yuhan Guo\footnote{Email address: guoyuhan25@mails.ucas.ac.cn}}
\author[3]{Dong Zhang\footnote{Email address: dongzhang@math.pku.edu.cn}}
\affil[1]{HLM, Academy of Mathematics and Systems Science, Chinese Academy
 of Sciences, Beijing, 100190, China }
 \affil[2]{School of Mathematical Sciences, University of
 Chinese Academy of Sciences, Beijing 100049, China}
\affil[3]{
School of Mathematical Sciences, Peking University, Beijing 100871, China}

\date{}

\begin{document}
\bibliographystyle{plain} 

\maketitle

\begin{abstract}
We show that for any connected graph $G$ with maximum degree $d\ge3$, the spectral gap from $0$ with respect to the adjacency matrix is at most $\sqrt{d-1}$, with equality if and only if $G$ is the incidence graph of a finite projective plane of order $d-1$; and for other cases, the bound $\sqrt{d-1}$ is improved to $\sqrt{d-2}$. This is a spectral gap version of a result by Mohar and Tayfeh-Rezaie. Moreover, for $d$-regular graphs with girth at least 7, the bound $\sqrt{d-2}$ is further improved to $\sqrt{d-c(d)}$ where $c(d)\ge 2$ and $\lim\limits_{d\to\infty}c(d)/d=(\sqrt{5}-1)/2$. 

A similar yet more subtle phenomenon involving the normalized Laplacian is also investigated, where we work on graphs of degrees $\ge d$ rather than $\le d$. We prove that for any graph $G$ with \emph{minimum} degree $d\ge 3$, the spectral gap from the value 1 with respect to the normalized Laplacian is at most $\sqrt{d-1}/d$, with equality if and only if $G$ is the incidence graph of a finite projective plane of order $d-1$. 
As an application, we provide a new sharp bound for the convergence rate of some eigenvalues of the Laplacian on the weighted neighborhood graphs introduced by Bauer and Jost.
\vspace{0.2cm}

\end{abstract}

\section{Introduction}

In this paper, we consider  linear operators associated to a connected, finite, simple graph $G=(V,E)$ on $N\geq 3$ vertices. For a vertex $v\in V$, we denote by $\deg v$ its degree, that is, the number of its neighbors. The degree matrix of $G$ is denoted by $\mathbf{D}(G):=\mathrm{diag}(\deg v_1,\cdots,\deg v_N)$, where $\{v_1,\cdots,v_N\}=V$. 
We use the notion $\mathbf{A}(G)$ to represent the adjacency matrix of $G$. For simplicity, we usually write $\mathbf{D}$ and $\mathbf{A}$ rather than $\mathbf{D}(G)$ and $\mathbf{A}(G)$. 
The normalized Laplacian of $G$ is simply defined by $\mathbf{L}:=\mathbf{I}-\mathbf{D}^{-1}\mathbf{A}$, where $\mathbf{I}$ is the identity matrix.  
We use $\sigma(\mathbf{A})$ and $\sigma(\mathbf{L})$ to denote the spectra of the adjacency matrix $\mathbf{A}$ and the normalized Laplacian $\mathbf{L}$, respectively. 

For the adjacency matrix, a systematic analysis of spectral gaps is presented by Koll\'ar and Sarnark \cite{Kollar20}, and they have identified many beautiful classes of graphs with a particular structure of their spectra and gap intervals. By considering other quantities involving adjacency eigenvalues (e.g., HL-index and the energy per vertex), there is rich information on their extremal graphs \cite{Mohar16,vanDam}. 

The normalized Laplacian generates random walks and diffusion processes on graphs. Previous works on the normalized Laplacian spectral gaps from 0 and from 2 have involved the Cheeger inequality \cite{Chung97} and the dual Cheeger inequality \cite{Bauer13,Trevisan}, with numerous significant applications in spectral clustering and MaxCut problems \cite{Trevisan}. 
It is also interesting that the normalized Laplacian spectral gap from 1 is closely related to the convergence rate of random walks on graphs \cite{Bauer13}.

Another central object in this paper is the finite projective plane \cite{Veblen,BBFT,Thompson}, which has been studied for more than a century and continues to attract widespread attention. As an important topic in incidence geometry, the study of finite projective planes is directly related to combinatorial designs \cite{Stinson}, and projective geometries \cite{Hirschfeld}. 
We shall give the history remark and detailed definition of finite projective planes in Section \ref{sec:projective}. 

The central discovery proposed in the paper is that: {\sl whether employing adjacency matrix or normalized Laplacian, the incidence graphs of finite projective planes are extremal graphs for the spectral gap from the average of eigenvalues}. 

Precisely, in the case of adjacency matrix, we characterize the extremal graphs for the spectral from 0 as follows.
\begin{theorem}\label{th:main}
Given $d\ge 3$, for any connected graph $G$ with \textbf{maximum} degree $\le d$, 
$$\min_{\lambda\in \sigma(\mathbf{A})}|\lambda|\le \sqrt{d-1}$$
with equality if and only if $G$ is the incidence graph of a finite projective plane of order $d-1$. \; 
Furthermore, if a connected graph $G$ has maximum degree $\le d$, and is not the incidence graph of a finite projective plane, then $$\min_{\lambda\in \sigma(\mathbf{A})}|\lambda|\le \sqrt{d-2}.$$
\end{theorem}

Theorem \ref{th:main} is a spectral gap reformulation of the significant related results \cite[Theorems 1, 3 and 4]{Mohar15} by Mohar and Tayfeh-Rezaie, where the distinctive feature of such reformulation is that we do not assume the \emph{bipartiteness}. When considering normalized Laplacian spectra instead of adjacency eigenvalues, we obtain the following spectral gap inequality, in which we work on \emph{gap from 1} rather than from 0, and the most fundamental difference lies in assuming the graph \emph{minimum} degree $\ge d$, rather than maximum degree $\le d$.


\begin{theorem}\label{th:main2}
Given $d\ge 3$, for any connected graph $G$ with \textbf{minimum} degree $\ge d$, 
$$\min\limits_{\lambda\in\sigma(\mathbf{L})}|\lambda-1|\le \frac{\sqrt{d-1}}{d}$$
with equality if and only if $G$ is the incidence graph of a finite projective plane of order $d-1$. 
\end{theorem}

The proof of Theorem \ref{th:main2} is more challenging than that of Theorem \ref{th:main}, since induced subgraphs do not have interlacing properties for normalized Laplacian eigenvalues. Our approach combines a nonregular-to-regular reduction technique and the interplay between spectra of 4-cycle free graphs and neighborhood graphs.

In fact, we can relate the Laplacian eigenvalues of a graph $G=(V,E)$ and those of its neighborhood graphs. The Laplacian spectra are essentially equivalent to each other, and therefore, eigenvalue estimates for a neighborhood graph can be translated into eigenvalue estimates for the original graph, and vice versa. 
Since the neighborhood graphs $G^{[l]}$ of order $l$ introduced in \cite{Bauer13} encode properties of random walks on $G$, asymptotic ones if $l\to \infty$, we thereby gain a new source of geometric intuition for obtaining eigenvalue estimates. Recall that the $l$-th normalized Laplacian $\mathbf{L}^{[l]}$ on the neighborhood graph $G^{[l]}$ satisfies
 $ \mathbf{L}^{[l]}= \mathbf{I}-(\mathbf{I}-\mathbf{L})^{l}
  $. Then, as a consequence of Theorem \ref{th:main2}, we obtain:
\begin{theorem}\label{theonei1}
For every connected graph  $G$ with {\sl minimum} degree $d\ge 3$, there is some eigenvalue $\lambda^{[l]}$ of $\mathbf{L}^{[l]}$ with
  \begin{equation*}
  |1-  \lambda^{[l]}|\le \Big(\frac{\sqrt{d-1}}{d}\Big)^l.
\end{equation*}
When $l=2k$ is an even number, the largest eigenvalue $\lambda^{[2k]}_N$ of   $\mathbf{L}^{[2k]}$ satisfies
  \begin{equation*}
    1- \frac{(d-1)^k}{d^{2k}}\le \lambda^{[2k]}_N\le 1,
  \end{equation*}
 and both bounds are sharp.
\end{theorem}

Note that in Theorems \ref{th:main} and \ref{th:main2}, the incidence graphs of finite projective planes are the common extremal graphs for the corresponding spectral gaps. Since the girth of the incidence graph of a finite projective plane is 6, it would be interesting to consider graphs with larger girths. 
\begin{prop}
\label{thm:girth7}
For any $d$-regular graph with girth at least 7, 
$$\min_{\lambda\in \sigma(\mathbf{A})}|\lambda|\le \sqrt{d-c_d} $$    
where $c_d$ is the positive root of the quadratic equation $t^2+(d-2)t-d(d-1)=0$.     
\end{prop}

Further useful formulations of Theorems \ref{th:main} and \ref{th:main2} and detailed results involving the normalized Laplacian are presented in Sections \ref{sec:main} and \ref{section:discuss}. 
Our results have fit into a larger picture: they have connections both with finite projective planes from combinatorial designs \cite{Hirschfeld,Stinson}, and with spectral extremal graph problems \cite{Brualdi,Nikiforov,Wilf}. Precisely, our results are spectral gap analogous to Mohar and Tayfeh-Rezaie's bounds on the HL-index \cite{Mohar16}, as well as van Dam, Haemers and Koolen's estimate of the energy per vertex \cite{vanDam}.  Moreover, our results are related to gap intervals 
and random walks on graphs, including 
Koll\'ar-Sarnak's theorem on the maximal gap interval for cubic graphs \cite{Kollar20}, as well as Bauer-Jost's Laplacian on neighborhood  graphs \cite{Bauer13}. 
We will explain these relations in Section \ref{sec:main} and Section \ref{section:discuss}.

\section{Preliminary and Main results}\label{sec:main}
Throughout the paper we fix a connected, finite, simple graph $G=(V,E)$ on $N\geq 3$ vertices. 
Given a vertex $v$, we let $\mathcal{N}(v):=\{w\in V:\{w,v\}\in E\}$ denote the neighborhood of $v$, i.e., the set of other vertices $w\sim v$ connected to $v$ by an edge. For convenience, we use the terminology $\mathcal{G}$ to express the set of all connected graphs with at least 3 nodes. Given $d\ge 2$, we use the following notions \[\mathcal{G}_{\ge d}=\{G\in \mathcal{G}:\deg(v)\ge d,\forall v\in V(G)\},\]
\[\mathcal{G}_{= d}=\{G\in \mathcal{G}:\deg(v)= d,\forall v\in V(G)\},\]
and
\[\mathcal{G}_{\le d}=\{G\in \mathcal{G}:\deg(v)\le d,\forall v\in V(G)\},\]
for the collections of connected graphs with minimum degree $\ge d$, connected $d$-regular graphs, and connected graphs with maximum degree $\le d$, respectively. 
In this section, we present a detailed review of concepts and results related to Theorems \ref{th:main} and \ref{th:main2}. We begin by introducing the incidence graphs of finite projective planes and the normalized Laplacian separately. 

\subsection{Finite projective planes and their incidence graphs}\label{sec:projective}

A finite projective plane is an incidence structure $(P,L,I)$ which consists of a finite set of points $P$, a finite set of lines $L$, and an incidence relation $I$ between the points and the lines that satisfy the following conditions:
\begin{enumerate}[({P}1)]
\item  Every two points are incident with a unique line.
\item Every two lines are incident with a unique point.
\item There are four points, no three collinear.
\end{enumerate}
A projective plane of \emph{order $n$} is a finite projective plane that has at least one line with exactly $n+1$ distinct points incident with it, where $n\ge 2$.

For what values of $n$ does a projective plane of order $n$ exist? 
This is a very fundamental question on finite projective planes. Veblen and Bussey proved that a finite projective plane exists when the order $n$ is a power of a prime, and they conjectured that these are the only possible projective planes \cite{Veblen}. This is one of the most important unsolved problems in combinatorics and some remarkable progresses are made by Bruck and Ryser \cite{BruckRyser}, and Lam \cite{Lam89}.

There are some extremal graph problems whose extremal graphs are the polarity graphs of finite projective planes \cite{Brown,ERS66,Furedi,Furedi13}, and the incidence graphs of finite projective planes \cite{DeWinter,Fiorini,Mohar,Mohar16,vanDam}. Since this paper focuses on incidence graphs, we recall the definition as follows.

\begin{definition}
The \emph{incidence graph} of a finite projective plane $(P,L,I)$ is a bipartite graph with bipartition $P$ and $L$, in which $p\in P$ and $l\in L$  are adjacent if and only if $p\in l$.  
\end{definition}

For example, the incidence graph of a finite projective plane of order $2$ is unique up to graph isomorphism, which is called the Heawood
 graph (see Figure \ref{fig:Heawood}).

\begin{prop}[\cite{Godsil}]
The eigenvalues of an incidence graph of a finite projective plane of order $n$ are $\pm (n+1)$, $\pm \sqrt{n}$, where the multiplicity of $\sqrt{n}$ (resp., $-\sqrt{n}$) is $n^2+n$. 
\end{prop}


\begin{figure}[h]
    \centering
\begin{tikzpicture}[scale=1.6
]

\foreach \i in {0,2,4,6,8,10,12} {
    \node (v\i) at ({\i*360/14}:1cm) {$\circ$};
}

\foreach \i in {1,3,5,7,9,11,13} {
    \node (v\i) at ({\i*360/14}:1cm) {$\bullet$};
}

\foreach \i in {0,2,4,6,8,10,12} {
    \pgfmathsetmacro{\j}{mod(\i+5,14)}
    \draw ({\i*360/14}:0.97cm) -- ({\j*360/14}:0.97cm);
}

\foreach \i in {0,...,13} {
    \pgfmathsetmacro{\j}{mod(\i+1,14)}
    \draw ({\i*360/14}:0.97cm) -- ({\j*360/14}:0.97cm);
}

\end{tikzpicture}
    \caption{The Heawood graph}
    \label{fig:Heawood}
\end{figure}
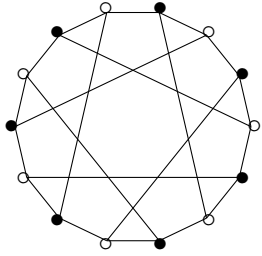


With the help of the incidence graphs of finite projective planes, 
Mohar \cite{Mohar16} establish the inequality $$\sup\limits_{G\in \mathcal{G}_{\le d}} R(G)\ge \sqrt{d-1}$$ where $R(G)$ indicates the HL-index of $G$. And for any \emph{bipartite} graph $G$ in $\mathcal{G}_{\le d}$, 
Mohar and Tayfeh-Rezaie further prove that if $R(G)> \sqrt{d-2}$ then $R(G)= \sqrt{d-1}$ and $G$ is the incidence graph of a projective plane of order $d-1$ (c.f. \cite[Theorem 3]{Mohar15}). To some extent, Theorem \ref{th:main} is an extension of this result. 
    
Coincidentally, van Dam, Haemers and Koolen \cite{vanDam} show that the energy per vertex of a $d$-regular graph is at most $$\frac{d+(d^2-d)\sqrt{d-1}}{d^2-d+1}$$
with equality if and only if the graph is the disjoint union of incidence graphs of projective planes of order $d-1$, or, in case $d=2$, the disjoint union of triangles and hexagons. 

Using the spectral gap from 0 instead of the HL-index and the energy per vertex, Theorem \ref{th:main} can be viewed as a spectral gap analogue of the results of Mohar \cite{Mohar16}, Mohar and Tayfeh-Rezaie \cite{Mohar15}, as well as van Dam, Haemers and Koolen \cite{vanDam}. We also point out in Remark \ref{remark:th1-equivalent} that Theorem \ref{th:main} is essentially a spectral gap reformulation of the result of Mohar and Tayfeh-Rezaie \cite{Mohar15}. 

Another spectral gap property regarding finite projective planes 
will be presented in Section \ref{sec:Laplacian}, in which we essentially use the normalized adjacency matrix $\mathbf{D}^{-1}\mathbf{A}$ rather than the adjacency matrix $\mathbf{A}$, but we would formulate the results in  terms of normalized Laplacian to fit the large picture on Laplacian spectral gap. 
\subsection{Normalized Laplacian and main results}\label{sec:Laplacian}
The normalized Laplacian $\mathbf{L}$ acting on a function $f:V\to \mathbb{R}$ is defined by
\begin{equation}\label{eq:defL}
    \mathbf{L} f(v)=f(v)-\frac{1}{\deg v} \sum_{w\sim v}f(w),
\end{equation}
that is, we subtract from the value of $f$ at $v$ the average of the values at its neighbors. This operator generates random walks and diffusion processes on graphs, and it was first systematically studied in \cite{Chung97}. The basic equality $\mathbf{L}+\mathbf{D}^{-1}\mathbf{A}=\mathbf{I}$ establishes the connection between the normalized Laplacian $\mathbf{L}$ and the normalized adjacency matrix $\mathbf{D}^{-1}\mathbf{A}$, the later of which is commonly used in graph convolutional neural networks \cite{Kipf}. Due to the simple relationship between the spectra of the two matrices, it suffices to work with one of them. We prefer to use the normalized Laplacian $\mathbf{L}$ because it has both geometric and combinatorial meanings. 

Denote by $\sigma(\mathbf{L})$ the spectrum of $\mathbf{L}$, and by $$\mathrm{gap}(G):=\min\limits_{\lambda\in\sigma(\mathbf{L})}|\lambda-1|$$
the spectral gap from 1.  
Given a subfamily $\mathcal{G}'\subset \mathcal{G}$, we use the notion
$$\textbf{gap}(\mathcal{G}'):=\sup_{G\in \mathcal{G}'}\mathrm{gap}(G)=\sup_{G\in \mathcal{G}'}\min_{\lambda\in\sigma(\mathbf{L})}|\lambda-1|.$$
Note that we always assume that $\mathcal{G}'$ is an infinite set. Denote by 
$$ \mathbf{Extreme}(\mathcal{G}')=\{G\in \mathcal{G}':\mathrm{gap}(G)=\textbf{gap}(\mathcal{G}')\} $$
the set of extremal graphs in $\mathcal{G}'$ for the Laplacian spectral gap from 1.  
We are in a position to present the extremal graph theory on the spectral gap, which is more subtle than Theorem \ref{th:main}. 
\begin{theorem}\label{thm:d-graph-projective}
Given $d\ge 3$, we have the following: 
\begin{itemize}
\item If there exists a finite projective plane of order $d-1$, then $$\mathbf{gap}(\mathcal{G}_{=d})=\mathbf{gap}(\mathcal{G}_{\ge d})=\frac{\sqrt{d-1}}{d}$$ 
and 
$\mathbf{Extreme}(\mathcal{G}_{\ge d})=\mathbf{Extreme}(\mathcal{G}_{=d})
$ is the set of incidence graphs of projective planes of order $d-1$. 
\item For any $G\in\mathcal{G}_{=d}$ other than incidence graphs of projective planes of order $d-1$, we have $$\mathrm{gap}(G)\le\frac{\sqrt{d-2}}{d}.$$

If we further assume that there exists \textbf{no} finite projective plane of order $d-1$, then for any $G\in\mathcal{G}_{\ge d}$, we have $\mathrm{gap}(G)<\frac{\sqrt{d-1}}{d}$ and 
$$\mathbf{gap}(\mathcal{G}_{=d})\le\frac{\sqrt{d-2}}{d}.$$
\end{itemize}
\end{theorem}

This result can also be viewed as a constrained version of the main theorem in \cite{JGT23} with additional minimum degree constraint. For example, $\mathbf{Extreme}(\mathcal{G}_{\ge 3})=\{\text{Heawood
 graph}\}$ (see Figure \ref{fig:Heawood}). It is interesting to notice that combining with the main result in \cite{JGT23}, the equality $\mathbf{Extreme}(\mathcal{G}_{\ge d})=\mathbf{Extreme}(\mathcal{G}_{=d})$ does not hold for $d=2$, as we have $\mathbf{Extreme}(\mathcal{G}_{= 2})\subsetneqq \mathbf{Extreme}(\mathcal{G}_{\ge 2})$ by the following result:

\begin{theorem}\label{thm:d=2}
If $d=2$, then $\mathbf{gap}(\mathcal{G}_{=2})=\mathbf{gap}(\mathcal{G}_{\ge 2})= \frac12$,  $\mathbf{Extreme}(\mathcal{G}_{=2})=\{\text{triangle,  hexagon}\}$, and $ \mathbf{Extreme}(\mathcal{G}_{\ge2})$ is the set of friendship graphs and book graphs (see Figure~\ref{fig:petal-book}).   
\end{theorem}

Theorems \ref{thm:d-graph-projective} and \ref{thm:d=2} indicate that case $d=2$ and case $d\ge 3$ have a very fundamental distinction. Also note that in some relevant results in \cite{Mohar16,Mohar15}, the \emph{interlacing property} is frequently used, 
and the \emph{bipartiteness} is required due to their approaches. However, unlike the adjacency matrix, the normalized Laplacian does not have such eigenvalue interlacing for induced subgraphs. 
To understand the difference of Theorems \ref{thm:d-graph-projective} and \ref{thm:d=2}, and to overcome the difficulties arising from the absence of bipartiteness and the interlacing property, we shall outline the proof. In fact, the proof contains two new strategies: 
\begin{itemize}
\item[\bf Phase 1.] Nonregular-to-regular reduction lemma: we use ideas from variational analysis and optimization to reduce the extremal graphs in the nonregular case to the regular case (see Lemma \ref{lemma:reduce-to-regular}). 
\item[\bf Phase 2.] Spectral interactions between 4-cycle free graphs and neighborhood graphs: we reveal a hidden relation between the normalized Laplacian of a regular 4-cycle free graph and adjacency matrix of its neighborhood graph, and we propose a deep study on the extremal graphs of the least adjacency eigenvalue of neighborhood graphs (see the proofs of Lemmas \ref{lemma:k} and \ref{lemma:phi(G)}).
    
\end{itemize}
We derive the proof by synthesizing these two strategies. 
Some of the lemmas have their own interests.

\begin{figure}[h]
    \centering
\begin{tabular}{ccc}
 \begin{tikzpicture}[scale=1.9]
 \draw (0,0) -- (30:1)-- (60:1)--(0,0)--(0:1)--(-30:1)--(0,0)--(0:1)--(0,0)--(-60:1)--(270:1)--(0,0)--(240:1)--(210:1)--(0,0)--(180:1)--(150:1)--(0,0);
 \node (0) at (0,0) {$\bullet$};
  \node (0) at (30:1) {$\bullet$};
   \node (0) at (60:1) {$\bullet$};
  \node (0) at (-30:1) {$\bullet$};
   \node (0) at (0:1) {$\bullet$};
   \node (0) at (-60:1) {$\bullet$};
  \node (0) at (270:1) {$\bullet$};
   \node (0) at (240:1) {$\bullet$};
  \node (0) at (210:1) {$\bullet$};
    \node (0) at (180:1) {$\bullet$};
  \node (0) at (150:1) {$\bullet$};
 \node (1) at (110:0.8) {$\cdot$};
  \node (1) at (100:0.8) {$\cdot$};
    \node (1) at (120:0.8) {$\cdot$};
 \end{tikzpicture}&~&  \begin{tikzpicture}[scale=1.2]
 \draw (0,0)--(2,1)--(2,2)--(0,3)--(1.6,2)--(1.6,1)--(0,0)--(1.2,1)--(1.2,2)--(0,3)--(0.8,2)--(0.8,1)--(0,0);
 \draw (0,0)--(-2,1)--(-2,2)--(0,3)--(-1.6,2)--(-1.6,1)--(0,0)--(-1.2,1)--(-1.2,2)--(0,3);
  \node (0) at (0,0) {$\bullet$};
  \node (0) at (2,1) {$\bullet$};
   \node (0) at (2,2) {$\bullet$};
  \node (0) at (0,3){$\bullet$};
   \node (0) at (1.6,2) {$\bullet$};
  \node (0) at (1.6,1) {$\bullet$};
   \node (0) at (1.2,1) {$\bullet$};
  \node (0) at (1.2,2) {$\bullet$};
  \node (0) at (0.8,2) {$\bullet$};
   \node (0) at (0.8,1) {$\bullet$};
  \node (0) at (-2,1) {$\bullet$};
    \node (0) at (-2,2) {$\bullet$};
   \node (0) at (-1.6,2) {$\bullet$};
  \node (0) at (-1.6,1) {$\bullet$};
   \node (0) at (-1.2,1) {$\bullet$};
  \node (0) at (-1.2,2) {$\bullet$};
 \node (1) at (-0.2,1.5) {$\cdots~~\cdots$};
 \end{tikzpicture} \\
  friendship graph   &~~~~~~~~~~~~& book graph 
\end{tabular}
    \caption{The friendship graphs and book graphs used in Theorem \ref{thm:d=2}}
    \label{fig:petal-book}
\end{figure}
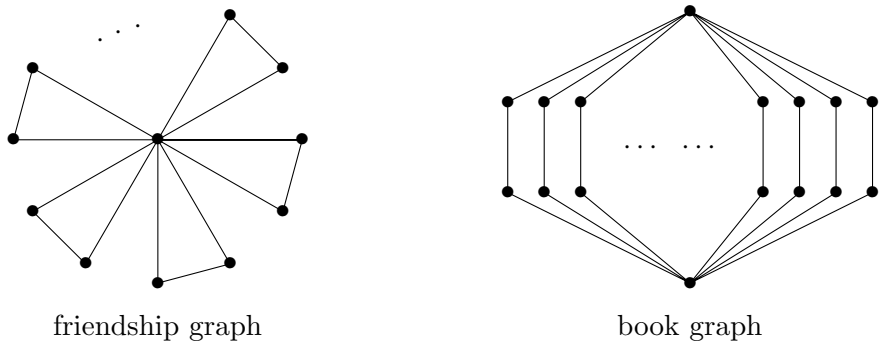

\section{Proof of the main results}\label{section:proof}

\subsection{Proof of Theorem \ref{thm:d-graph-projective} and Theorem \ref{th:main2}}
For a subset $S\subset V$, we use $|S|$ to denote the number of elements in $S$. 
Before proving Theorem \ref{thm:d-graph-projective}, we first establish a series of auxiliary lemmas.

\begin{lemma}\label{coro:1}
Let $G=(V,E)$ be a graph in $\mathcal{G}$. Then 
$$\big(\mathrm{gap}(G)\big)^2=\min_{f:V\rightarrow{\mathbb{R}},f\ne\mathbf{0}}\frac{\sum\limits_{u,v\in V}\sum\limits_{w\in \mathcal{N}(u)\cap\mathcal{N}(v)}\frac{f(u)f(v)}{\deg (w)\sqrt{\deg (u)\deg (v)}}}{\sum\limits_{w\in V}f(w)^2}.$$
Particularly, if $G$ is $d$-regular, then $$\big(\mathrm{gap}(G)\big)^2=\frac{1}{d^2}\min_{f:V\rightarrow{\mathbb{R}},f\ne\mathbf{0}}\frac{\sum\limits_{u,v\in V}{ |\mathcal{N}(u)\cap\mathcal{N}(v)|f(u)f(v)}}{\sum\limits_{w\in V}f(w)^2}.$$
\end{lemma}

\begin{proof}
Let $\lambda_1,\cdots,\lambda_N$ be the eigenvalues of $\mathbf{L}$. Then $(1-\lambda_1)^2,\cdots,(1-\lambda_N)^2$ are the eigenvalues of $\mathbf{M}=(\mathbf{I}-\mathbf{D}^\frac{1}{2}\mathbf{L} \mathbf{D}^{-\frac{1}{2}})^2$, where $\mathbf{D}=\mathrm{diag}(\deg   v_1,\cdots,\deg   v_N)$ is the diagonal matrix consisting of the degrees. Therefore, $\min_{\lambda\in\sigma(\mathbf{L})}|1-\lambda|^2$ is the least eigenvalue of $M$. We notice that the matrix entries of $\mathbf{D}^\frac{1}{2}\mathbf{L} \mathbf{D}^{-\frac{1}{2}}$ are $$(\mathbf{D}^\frac{1}{2}\mathbf{L} \mathbf{D}^{-\frac{1}{2}})_{uv}=\begin{cases}1&u=v\\-\frac{1}{\sqrt{\deg(u)\deg(v)}}&u\sim v\\0&\text{otherwise}\end{cases}$$
where $u,v\in V:=\{1,\cdots,N\}$. 
Thus, the entries of the matrix $M$ are $$M_{uv}=\sum_{w\in \mathcal{N}(u)\cap\mathcal{N}(v)}\frac{1}{\deg (w)\sqrt{\deg (u)\deg (v)}}$$
and the least eigenvalue of $M$ can be expressed as $$
\lambda_{\min}(M):=\min_{f:V\rightarrow{\mathbb{R}},f\ne\mathbf{0}}\frac{\sum_{u,v\in V}\sum_{w\in \mathcal{N}(u)\cap\mathcal{N}(v)}\frac{f(u)f(v)}{\deg (w)\sqrt{\deg (u)\deg (v)}}}{\sum_{w\in V}f(w)^2}.$$
Lemma \ref{coro:1} then follows from $\big(\mathrm{gap}(G)\big)^2=\min\limits_{\lambda\in\sigma(\mathbf{L})}|1-\lambda|^2=\lambda_{\min}(M)$. 
\end{proof}

\begin{lemma}\label{lemma:gap-inequality}
Let $G=(V,E)$ be a graph in $\mathcal{G}$. 
For any distinct $u,v\in V$, 
\begin{equation}\label{eq:repeat-key-opti-d}
\sum_{w\in \mathcal{N}(u)\triangle \mathcal{N}(v)}\left(\frac{1}{\deg w}- \big(\mathrm{gap}(G)\big)^2 \right)\ge  2 |\mathcal{N}(u)\cap \mathcal{N}(v)| \big(\mathrm{gap}(G)\big)^2 
\end{equation}
where $\mathcal{N}(u)\triangle \mathcal{N}(v)
$ is the symmetric difference of $\mathcal{N}(u)$ and $\mathcal{N}(v)$.
\end{lemma}

\begin{proof}
Taking a test function $f_{u,v}:V\to \mathbb{R}$ defined by
$$f_{u,v}(x)=\begin{cases}
\sqrt{\deg u},&\text{ if }x=u\\
-\sqrt{\deg v},&\text{ if }x=v\\
0,&\text{ otherwise}
\end{cases}$$
we have by Lemma \ref{coro:1} 
$$\frac{\sum_{u',v'\in V}\sum_{w\in \mathcal{N}(u')\cap\mathcal{N}(v')}\frac{f_{u,v}(u')f_{u,v}(v')}{\deg (w)\sqrt{\deg (u')\deg (v')}}}{\sum_{w\in V}f_{u,v}(w)^2}\ge \big(\mathrm{gap}(G)\big)^2.$$
Simplifying the left hand side as
$$\frac{\sum\limits_{w\in \mathcal{N}(u)}\frac{1}{\deg w}+\sum\limits_{w\in\mathcal{N}(v)}\frac{1}{\deg w}-2\sum\limits_{w\in\mathcal{N}(u)\cap\mathcal{N}(v)}\frac{1}{\deg w}}{\deg u+\deg v}=\frac{\sum\limits_{w\in \mathcal{N}(u)\triangle \mathcal{N}(v)}\frac{1}{\deg w}}{\deg u+\deg v}$$
we derive 
$$\sum\limits_{w\in \mathcal{N}(u)\triangle \mathcal{N}(v)}\frac{1}{\deg w}\ge (\deg u+\deg v) \big(\mathrm{gap}(G)\big)^2$$
which yields \eqref{eq:repeat-key-opti-d} by noting that $\deg u+\deg v=|\mathcal{N}(u)\triangle \mathcal{N}(v)|+2|\mathcal{N}(u)\cap \mathcal{N}(v)|$.
\end{proof}

\begin{lemma}\label{lemma:reduce-to-regular}
Given $d\ge 3$ and  $G\in\mathcal{G}_{\ge d}$, if $\mathrm{gap}(G)\ge \frac{\sqrt{d-1}}{d}$, then $\mathrm{gap}(G)=\frac{\sqrt{d-1}}{d}$ and $G$ is $d$-regular. 
\end{lemma}

\begin{proof}
We complete the proof by several claims. 

Claim 1: For any distinct vertices $u$ and $v$ with $\mathcal{N}(u)\cap \mathcal{N}(v)\ne\emptyset$, there holds $$\big|\big\{w\in \mathcal{N}(u)\triangle \mathcal{N}(v):\deg w\in\{d,d+1\}\big\}\big|\ge 2d-2.$$

Proof of Claim 1: Since $|\mathcal{N}(u)\cap \mathcal{N}(v)|\ge 1$, it follows from the inequality \eqref{eq:repeat-key-opti-d} in Lemma \ref{lemma:gap-inequality} that 
\begin{equation}\label{eq:repeat-key-opti-d2} 
\sum_{w\in \mathcal{N}(u)\triangle \mathcal{N}(v)}\left(\frac{1}{\deg w}-\frac{d-1}{d^2} \right)\ge 2  \frac{d-1}{d^2}.      
\end{equation}
Note that when $\deg w\ge d+2$, we have $ \frac{1}{\deg w}-\frac{d-1}{d^2}\le \frac{1}{d+2}-\frac{d-1}{d^2}=\frac{2-d}{d^2(d+2)}<0$. 
Suppose the contrary, that Claim 1 does not hold. Then \begin{align*}
\sum_{w\in \mathcal{N}(u)\triangle \mathcal{N}(v)}\left(\frac{1}{\deg w}-\frac{d-1}{d^2} \right)&\le\sum_{\substack{w\in \mathcal{N}(u)\triangle \mathcal{N}(v)\\ \deg w\in\{d,d+1\}}}\left(\frac{1}{\deg w}-\frac{d-1}{d^2} \right)\\&\le (2d-3) \left(\frac{1}{d}-\frac{d-1}{d^2}\right)<
2\frac{d-1}{d^2},
\end{align*}
which contradicts \eqref{eq:repeat-key-opti-d2}. 

Claim 2: For any distinct vertices $u$ and $v$ such that $\mathcal{N}(u)\cap \mathcal{N}(v)\ne\emptyset$, if  $$\big|\big\{w\in \mathcal{N}(u)\triangle \mathcal{N}(v):\deg w\in\{d,d+1\}\big\}\big|\le 2d$$
then $|\{w\in \mathcal{N}(u)\triangle \mathcal{N}(v):\deg w=d\}|\ge 2d-3$. 

Proof of Claim 2: If $|\{w\in \mathcal{N}(u)\triangle \mathcal{N}(v):\deg w=d\}|\le 2d-4$, then \begin{align*}
&\sum_{w\in \mathcal{N}(u)\triangle \mathcal{N}(v)}\left(\frac{1}{\deg w}-\frac{d-1}{d^2} \right)\\
&\le\sum_{\substack{w\in \mathcal{N}(u)\triangle \mathcal{N}(v)\\ \deg w=d}}\left(\frac{1}{\deg w}-\frac{d-1}{d^2} \right)+\sum_{\substack{w\in \mathcal{N}(u)\triangle \mathcal{N}(v)\\ \deg w=d+1}}\left(\frac{1}{\deg w}-\frac{d-1}{d^2} \right)\\
&\le (2d-4) \left(\frac{1}{d}-\frac{d-1}{d^2}\right)+ 2d \left(\frac{1}{d+1}-\frac{d-1}{d^2}\right)\\
&=\frac{2d-4}{d^2}+\frac{2}{d(d+1)}<\frac{2d-4}{d^2}+\frac{2}{d^2}
=2\cdot \frac{d-1}{d^2},
\end{align*}
which contradicts \eqref{eq:repeat-key-opti-d2}. 

\vspace{0.2cm}

\begin{figure}[h]
    \centering
    \begin{tikzpicture}[scale=1.6]
\draw (0,0)--(1,0)--(2,0)--(3,0);
\draw[dashed] (0.2,0.9)--(0,0)--(-1,0);\draw
(1.8,0.9)--(2,0);
\draw (3.2,0.9)--(3,0)--(4,0);
\draw (4.2,0.9)--(4,0)--(5,0)--(6,0);
\node (v) at  (0,0) {$\bullet$};
\node (w) at  (1,0) {$\bullet$};
\node (u) at  (2,0) {$\bullet$};
\node (w1) at  (3,0) {$\bullet$};
\node (w2) at (1.8,0.9){$\bullet$};
\node (x) at  (4,0) {$\bullet$};
\node (x1) at  (5,0) {$\bullet$};
\node (z1) at  (6,0) {$\bullet$};
\node (y) at  (3.2,0.9) {$\bullet$};
\node (x2) at  (4.2,0.9) {$\bullet$};
\node (v) at  (-0.1,0.2) {$v$};
\node (w) at  (0.8,0.2) {$w$};
\node (u) at  (2.1,0.22) {$u$};
\node (w1) at  (2.9,0.2) 
{$w_1$};\node (w2) at (2,0.8){$w_2$};\node (x) at  (3.9,0.2) {$x$};\node (u) at  (3,0.8) {$y$};
\node (x1) at  (4.9,0.2) {$x_1$};\node (x2) at  (4,0.8) {$x_2$};
\node (z1) at  (5.9,0.2) {$z_1$};
\end{tikzpicture}
    \caption{The vertices in the proof of Lemma \ref{lemma:reduce-to-regular}}
    \label{fig:reduace-to-regular}
\end{figure}
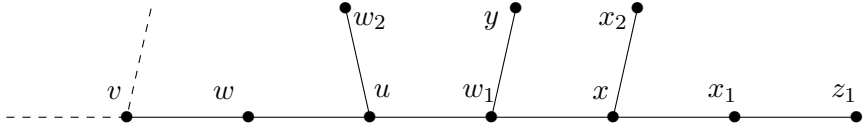

We are in a position to prove Lemma \ref{lemma:reduce-to-regular} with an illustration in Figure \ref{fig:reduace-to-regular}. 
For any path $v\sim w\sim u$, applying Claim 1 to $u$ and $v$, we either have  $$|\{w\in \mathcal{N}(u)\setminus \mathcal{N}(v):\deg w\in\{d,d+1\}\}|\ge d-1\ge 2$$ or 
$$|\{w\in \mathcal{N}(v)\setminus \mathcal{N}(u):\deg w\in\{d,d+1\}\}|\ge d-1\ge 2.$$ Without loss of generality, we may assume that there are two vertices $w_1$ and $w_2$ in $\mathcal{N}(u)\setminus \mathcal{N}(v)$ with $\deg w_1,\deg w_2\in\{d,d+1\}$. Then, we have the path $w_1\sim u\sim w_2$ in $G$, and $| \mathcal{N}(w_1)\triangle \mathcal{N}(w_2)|\le \deg w_1+\deg w_2-2\le 2d$. Applying Claim 2 to $w_1$ and $w_2$, we obtain $|\{w\in \mathcal{N}(w_1)\triangle \mathcal{N}(w_2):\deg w=d\}|\ge 2d-3\ge 3$. Hence, without loss of generality, we can assume that there are two vertices $x$ and $y$ in $\mathcal{N}(w_1)\setminus \mathcal{N}(w_2)$ with $\deg x=\deg y=d$.  Applying Claim 1 to $x$ and $y$, we derive 
\begin{align*}
\sum_{w\in \mathcal{N}(x)\triangle \mathcal{N}(y)}\left(\frac{1}{\deg w}-\mathrm{gap}(G) \right)&\le\sum_{\substack{w\in \mathcal{N}(x)\triangle \mathcal{N}(y)\\ \deg w\in\{d,d+1\}}}\left(\frac{1}{\deg w}-\frac{d-1}{d^2} \right)\\&\le |\mathcal{N}(x)\triangle \mathcal{N}(y)|
\left(\frac{1}{d}-\frac{d-1}{d^2} \right)
\\&\le (2d-2) \left(\frac{1}{d}-\frac{d-1}{d^2}\right)=
2\frac{d-1}{d^2}
\\&\le2 |\mathcal{N}(x)\cap \mathcal{N}(y)| \big(\mathrm{gap}(G)\big)^2,
\end{align*}
and again, combining this with the inequality \eqref{eq:repeat-key-opti-d}, there actually holds the equality, which implies that $|\mathcal{N}(x)\cap \mathcal{N}(y)|=1$, and $\deg w=d$ for any $w\in \mathcal{N}(x)\triangle \mathcal{N}(y)$, and $\mathrm{gap}(G)=\frac{\sqrt{d-1}}{d}$. The above reasoning process indeed proves the following claim. 

\vspace{0.1cm}

Claim 3: For any distinct vertices $x'$ and $y'$ such that $\mathcal{N}(x')\cap \mathcal{N}(y')\ne\emptyset$, if $\deg x'=\deg y'=d$, then $|\mathcal{N}(x')\cap \mathcal{N}(y')|=1$, $|\mathcal{N}(x')\setminus \mathcal{N}(y')|\ge d-1\ge 2$, $|\mathcal{N}(y')\setminus \mathcal{N}(x')|\ge d-1\ge 2$, and for any $w\in \mathcal{N}(x')\triangle \mathcal{N}(y')$, $\deg w=d$.

\vspace{0.1cm}

Now, applying Claim 3 to $x$ and $y$, we have for any distinct vertices $x_1,x_2\in \mathcal{N}(x)\setminus \mathcal{N}(y)$, $\deg x_1=\deg x_2=d$, and then applying Claim 3 to $x_1$ and $x_2$, we can take $z_1\in \mathcal{N}(x_1)\setminus \mathcal{N}(x)$ with $\deg z_1=d$. Clearly, $w_1\not\sim z_1$, otherwise, $\mathcal{N}(z_1)\cap \mathcal{N}(x)$ contains at least two distinct vertices $x_1$ and $w_1$, which contradicts Claim 3.

Since $w_1\in \mathcal{N}(z_1)\triangle \mathcal{N}(x)$, we can apply Claim 3 to $z_1$ and $x$ to derive $\deg w_1=d$. Repeating the process, we apply  Claim 3 to $x_1$ and $w_1$. Then it follows from $u\in \mathcal{N}(x_1)\triangle \mathcal{N}(w_1)$ that $\deg u=d$. Again, applying Claim 3 to $x$ and $u$, we have $\deg w=d$; and applying  Claim 3 to $w_1$ and $w$, we have $\deg v=d$. 

\vspace{0.1cm} 

Note that, we start with any path $u\sim w\sim v$ in $G$ and recursively derive that $\deg u=\deg w=\deg v=d$. By the arbitrariness of $u$ and $w$ and $v$, 
we indeed derive that every vertex has degree $d$, meaning that $G$ is $d$-regular.
\end{proof}

Lemma \ref{lemma:reduce-to-regular} indicates that we can reduce the non-regular case to regular case for the extremal graphs with the largest spectral gap $\frac{\sqrt{d-1}}{d}$.

\begin{lemma}\label{lemma:=d-no-4-cycle}
Let $G$ be a $d$-regular graph. If $\mathrm{gap}(G)>\frac{\sqrt{d-2}}{d}$, then $G$ is 4-cycle free, 
in other words, for any $u,v\in V(G)$ with $u\ne v$, there always holds $|\mathcal{N}(u)\cap\mathcal{N}(v)|\le1$.
\end{lemma}

\begin{proof}
Suppose the contrary, that $|\mathcal{N}(u)\cap\mathcal{N}(v)|\ge2$, where $u$ and $v$ are distinct vertices of $G$. By Lemma \ref{coro:1}, we can take a test function $f$ defined by $f(u)=1$, $f(v)=-1$, and $f(x)=0$ whenever $x\not\in\{u,v\}$. Then we have
$$\big(\mathrm{gap}(G)\big)^2\le\frac{2d-2|\mathcal{N}(u)\cap\mathcal{N}(v)|}{2d^2}\le\frac{2d-4}{2d^2}=\frac{d-2}{d^2}$$
which implies $\mathrm{gap}(G)\le\frac{\sqrt{d-2}}{d}$, a contradiction.
\end{proof}

Now we introduce the (unweighted) neighborhood graph which can also be obtained from $G^{[2]}$ by resetting all the weights to 1 (see \cite{Bauer13,Kim02,Roberts}): 
\begin{definition}\label{def:construct-new-graph}
Given a connected graph $G=(V,E)$, we define $\phi(G)$ as follows: 
\begin{itemize}
\item the vertex set of $\phi(G)$ is the same to that of $G$, i.e.,  $V(\phi(G)):=V$
\item two vertices $u$ and $v$ are adjacent in $\phi(G)$ if and only if they have common neighbors in $G$, that is, 
$$E(\phi(G)):=\big\{\{u,v\}\subset V:u\ne v\text{ and }|\mathcal{N}(u)\cap\mathcal{N}(v)|\ge1\big\}.$$
\end{itemize}
We call $\phi(G):=(V,E(\phi(G)))$ the (unweighted) \emph{neighborhood graph} of $G$.
\end{definition}

\begin{lemma}\label{lemma:phi}
Let $G$ be a 4-cycle free $d$-regular graph. Then every vertex in $\phi(G)$ has exactly $(d^2-d)$ neighbors, that is, $|\mathcal{N}_{\phi(G)}(v)|=d^2-d$, where $\mathcal{N}_{\phi(G)}(v)$ denotes the neighborhood of a vertex $v$ in $\phi(G)$.
\end{lemma}
Since Lemma \ref{lemma:phi} is very elementary and does not involve any information on spectral gaps, we put its proof in the Appendix. 

\begin{lemma}\label{lemma:k}
Let $G$ be a $d$-regular graph. If $$\mathrm{gap}(G)>\frac{\sqrt{d-2}}{d},$$
then $\phi(G)$ is the disjoint union of complete graphs.
\end{lemma}
\begin{proof} 
By Lemma \ref{lemma:=d-no-4-cycle}, for any $\{u,v\}\in E(\phi(G))$, $|\mathcal{N}(u)\cap\mathcal{N}(v)|=1$, and for any distinct vertices $u$ and $v$ with $\{u,v\}\not\in E(\phi(G))$, $|\mathcal{N}(u)\cap\mathcal{N}(v)|=0$. Thus,
$$\frac{\sum_{u,v\in V}{ |\mathcal{N}(u)\cap\mathcal{N}(v)|f(u)f(v)}}{\sum_{w\in V}f(w)^2}=d+2\frac{\sum_{\{u,v\}\in E(\phi(G))}{f(u)f(v)}}{\sum_{w\in V}f(w)^2}$$

It then follows from Lemma \ref{coro:1} that
\begin{align*}
\big(\mathrm{gap}(G)\big)^2&=\frac{1}{d^2}\min_{f:V\rightarrow{\mathbb{R}},f\ne\mathbf{0}}\frac{\sum_{u,v\in V}{ |\mathcal{N}(u)\cap\mathcal{N}(v)|f(u)f(v)}}{\sum_{w\in V}f(w)^2}
\\&=\frac{1}{d^2}\left(d+\min_{f:V\rightarrow{\mathbb{R}},f\ne\mathbf{0}}\frac{2\sum_{\{u,v\}\in E(\phi(G))}{f(u)f(v)}}{\sum_{w\in V}f(w)^2}\right)
\\&=\frac{d+\lambda_{\min}(\mathbf{A}(\phi(G)))}{d^2}
\end{align*}
where $\lambda_{\min}(\mathbf{A}(\phi(G)))$ represents the smallest eigenvalue of the adjacency matrix of $\phi(G)$. 
The condition $\mathrm{gap}(G)>\frac{\sqrt{d-2}}{d}$ implies $$\frac{d-2}{d^2}<\big(\mathrm{gap}(G)\big)^2=\frac{d+\lambda_{\min}(\mathbf{A}(\phi(G)))}{d^2}$$
that is, $\lambda_{\min}(\mathbf{A}(\phi(G)))>-2$. 
By Lemma \ref{lemma:phi}, $\phi(G)$ is a $(d^2-d)$-regular graph which may not be connected. Therefore, every connected component of $\phi(G)$ is a connected regular graph whose least adjacency eigenvalue is greater than $-2$. 

Recall the important result by Doob and Cvetkovi\'c \cite[Theorem 2.5]{Doob} (see also Corollary 2.3.22 in \cite{Cvetkovicbook}) saying that any connected regular graph with least adjacency eigenvalue greater than $-2$ must be a complete graph or an odd cycle. 

We then claim that every connected component of $\phi(G)$ is a complete graph or an odd cycle. However, since the degree of any vertex of $\phi(G)$ is constant $d^2-d\ge 6$, no connected component can be an odd cycle. In consequence, every connected component of $\phi(G)$ is a complete graph.
 \end{proof}

\begin{lemma}\label{lemma:=phi}Let $G$ be a $d$-regular connected graph. If $\mathrm{gap}(G)>\frac{\sqrt{d-2}}{d}$, 
then $\phi(G)$ must be a complete graph of order $d^2-d+1$, or the disjoint union of two complete graphs of order $d^2-d+1$.
\end{lemma}

\begin{proof}
By Lemma \ref{lemma:k}, each connected component of $\phi(G)$ must be a complete graph. And by Lemma \ref{lemma:phi}, each of these complete graphs must be of order $d^2-d+1$.   Without loss of generality, we assume $\phi(G)=K_1\bigsqcup\cdots\bigsqcup K_m$ with each $K_t$ being a complete graph of order $d^2-d+1$. For any edge $\{u,v\}\in E$, suppose that $u\in K_t$ and $v\in K_s$ for some $t,s\in \{1,\cdots,m\}$. 

Case $t=s$: We first claim that for any $w\sim u$, $w\in V(K_t)$. If not, then $u\in\mathcal{N}(v)\cap\mathcal{N}(w)$ implying that $\{v,w\}\in E(\phi(G))$ which contradicts $v\in V(K_s)=V(K_t)$ and $w\notin V(K_t)$. For the same reason, every $w'\sim v$ satisfies $w'\in K_t$. Then, the connectedness of $G$ implies that any vertex lies in $K_t$. Hence, $V(K_t)=V$ and $m=1$.

Case $t\ne s$: We claim that for any $w\sim u$, $w\in V(K_s)$. Suppose the contrary, that there is $w\sim u$ with $w\notin V(K_s)$. Then $u\in\mathcal{N}(v)\cap\mathcal{N}(w)$ and thus $\{v,w\}\in E(\phi(G))$, but this  contradicts $v\in V(K_s)$ and $w\notin V(K_s)$. 
Similarly, for any $w'\sim v$, $w'\in V(K_t)$. Finally, by the connectedness of $G$, it is not difficult to see that any edge of $G$ has an endpoint in $K_t$ and the other endpoint in $K_s$. Therefore, $\phi(G)$ is the disjoint union of two complete graphs of order $d^2-d+1$, and in this case, we have $m=2$.
\end{proof}


\begin{lemma}\label{lemma:phi(G)}Let $G$ be a 4-cycle free $d$-regular connected graph. Assume that $\phi(G)$ is a complete graph of order $d^2-d+1$, or the disjoint union of two complete graphs of order $d^2-d+1$. Then $\mathrm{gap}(G)=\frac{\sqrt{d-1}}{d}$.
\end{lemma}
\begin{proof}
The 4-cycle free condition means that for any distinct vertices $u,v\in V$,  $|\mathcal{N}(u)\cap\mathcal{N}(v)|\le1$. Thus $\{u,v\}\in E(\phi(G))$ if and only if $|\mathcal{N}(u)\cap\mathcal{N}(v)|=1$. 

If $\phi(G)$ is a complete graph of order $d^2-d+1$, we can then assume $V(\phi(G))=V=\{u_1,\cdots,u_{d^2-d+1}\}$. By Lemma \ref{coro:1}, we have
\begin{align*}
    \big(\mathrm{gap}(G)\big)^2&=\frac{1}{d^2}\min_{f:V\rightarrow{\mathbb{R}},f\ne\mathbf{0}}\frac{\sum_{u,v\in V}{ |\mathcal{N}(u)\cap\mathcal{N}(v)|f(u)f(v)}}{\sum_{w\in V}f(w)^2}\\&=\frac{1}{d^2}\min_{f:V\rightarrow{\mathbb{R}},f\ne\mathbf{0}}\frac{d\sum_{i}{ f(u_i)^2}+2\sum_{i<j}f(u_i)f(u_j)}{\sum_{i}f(u_i)^2}\\&=\frac{1}{d^2}\Big(d-1+\min_{f:V\rightarrow{\mathbb{R}},f\ne\mathbf{0}}\frac{\big(\sum_{i}{ f(u_i)}\big)^2}{\sum_{i}f(u_i)^2}\Big)\\&=\frac{d-1}{d^2}.
\end{align*}
In consequence, $\mathrm{gap}(G)=\frac{\sqrt{d-1}}{d}$.

If $\phi(G)$ is the disjoint union of two complete graphs of order $d^2-d+1$, we can similarly assume that the vertex sets of the two complete graphs are $\{u_1,\cdots,u_{d^2-d+1}\} $ and $\{v_1,\cdots,v_{d^2-d+1}\}$, respectively. By Lemma \ref{coro:1}, we have
\begin{align*}
    \big(\mathrm{gap}(G)\big)^2&=\frac{1}{d^2}\min_{f:V\rightarrow{\mathbb{R}},f\ne\mathbf{0}}\frac{\sum_{u,v\in V}{ |\mathcal{N}(u)\cap\mathcal{N}(v)|f(u)f(v)}}{\sum_{w\in V}f(w)^2}\\&=\frac{1}{d^2}\min_{f:V\rightarrow{\mathbb{R}},f\ne\mathbf{0}}\frac{d\sum\limits_{i}{ \big(f(u_i)^2+f(v_i)^2\big)}+2\sum\limits_{i<j}\big(f(u_i)f(u_j)+f(v_i)f(v_j)\big)}{\sum\limits_{i}(f(u_i)^2+f(v_i)^2)}\\&=\frac{1}{d^2}\Big(d-1+\min_{f:V\rightarrow{\mathbb{R}},f\ne\mathbf{0}}\frac{\big(\sum_{i}{ f(u_i)}\big)^2+\big(\sum_{i}{ f(v_i)}\big)^2}{\sum_{i}\big(f(u_i)^2+f(v_i)^2\big)}\Big)\\&=\frac{d-1}{d^2}.
\end{align*}
Consequently, $\mathrm{gap}(G)=\frac{\sqrt{d-1}}{d}$.
\end{proof}

It follows from Lemmas \ref{lemma:=d-no-4-cycle}, \ref{lemma:=phi} and \ref{lemma:phi(G)} that for a connected $d$-regular graph $G$, the following conditions are equivalent:
\begin{itemize}
\item $\mathrm{gap}(G)=\frac{\sqrt{d-1}}{d}$
\item $\mathrm{gap}(G)>\frac{\sqrt{d-2}}{d}$
\item $G$ is 4-cycle free, and $\phi(G)$ is the complete graph of order $d^2-d+1$, or the disjoint union of two complete graphs, each of them has order $d^2-d+1$.
\end{itemize}

We shall start from the last condition to explore more on the combinatorial characterization of $G$. 

\begin{lemma}\label{lemma:d=2}
Let $G$ be a 4-cycle free $d$-regular connected graph. Assume that $\phi(G)$ is a complete graph of order $d^2-d+1$. Then $d=2$.
\end{lemma}
\begin{proof}
Following the proof of Lemma \ref{lemma:phi(G)}, if $\phi(G)$ is a complete graph on the vertices $u_1,\cdots,u_{d^2-d+1}$, then 
$$\big(\mathrm{gap}(G)\big)^2=\frac{1}{d^2}\Big(d-1+\min_{f:V\rightarrow{\mathbb{R}},f\ne\mathbf{0}}\frac{(\sum_{i}{ f(u_i)})^2}{\sum_{i}f(u_i)^2}\Big)=\frac{d-1}{d^2}.$$
Note that the linear subspace $\{f:\sum_i f(u_i)=0\}$ is of dimension $(d^2-d+1)-1=d^2-d$. Therefore, the multiplicity of the eigenvalue $\frac{d-1}{d^2}$ of the matrix $\mathbf{M}:=(\mathbf{I}-\mathbf{D}^\frac{1}{2}\mathbf{L} \mathbf{D}^{-\frac{1}{2}} )^2$ is $d^2-d$. 

This implies that both $1-\frac{\sqrt{d-1}}{d}$ and $1+\frac{\sqrt{d-1}}{d}$ are eigenvalues of $\mathbf{L}$, and the sum of their multiplicities is $d^2-d$. Without loss of generality, we may assume that the multiplicity of the eigenvalue $1-\frac{\sqrt{d-1}}{d}$ is $r$. Together with the fact that $0$ is always an eigenvalue of $\mathbf{L}$ with multiplicity one, we finally determine all the eigenvalues of $\mathbf{L}$, which are $0$, $1-\frac{\sqrt{d-1}}{d}$, $1+\frac{\sqrt{d-1}}{d}$, with their multiplicities $1$, $r$ and $d^2-d-r$, respectively. 
Note that the sum of all the eigenvalues of 
$\mathbf{L}$ is the number of vertices, that is, $d^2-d+1$. Therefore, 
$$r(1-\frac{\sqrt{d-1}}{d})+(d^2-d-r)(1+\frac{\sqrt{d-1}}{d})=d^2-d+1$$
which reduces to $$(d^2-d-2r)\frac{\sqrt{d-1}}{d}=1.$$
Thus, 
$\sqrt{d-1}$ is a rational number, and hence $d=m^2+1$ for some positive integer $m$. We then obtain $$(d^2-d-2r)m=m^2+1$$ which yields $m|1$, and in consequence, $m=1$, i.e., $d=2$.
\end{proof}

We are in a position to prove Theorem \ref{thm:d-graph-projective}. Suppose that $\phi(G)$ is the disjoint union of two complete graphs, in which one of them has the vertex set $\{u_1,\cdots,u_{d^2-d+1}\}$, and the other has the vertex set $\{v_1,\cdots,v_{d^2-d+1}\}$. Based on the proof of Lemma \ref{lemma:=phi}, all the edges are of the form $\{u_i,v_j\}$, i.e., $G$ is a bipartite graph. 
Since $G$ is 4-cycle free, we have $|\mathcal{N}(v_i)\cap\mathcal{N}(v_j)|\le1$ whenever  $i\ne j$. 
On the other hand, since $\{v_i,v_j\}\in E(\phi(G))$, we have $|\mathcal{N}(v_i)\cap\mathcal{N}(v_j)|\ge1$. Therefore, $|\mathcal{N}(v_i)\cap\mathcal{N}(v_j)|=1$, and similarly, $|\mathcal{N}(u_i)\cap\mathcal{N}(u_j)|=1$, whenever $i\ne j$. By viewing $\{u_1,\cdots,u_{d^2-d+1}\}$ as points and regarding $\{v_1,\cdots,v_{d^2-d+1}\}$ as lines, we get a finite projective plane of order $d-1$. And it is easy to check that $G$ is indeed the incidence graph of such a finite projective plane. 

For the case that $\phi(G)$ is a complete graph of order $d^2-d+1$, it follows from Lemma \ref{lemma:d=2} that $d=2$, and in this case, $G$ must be a cycle of order 3 or 6. 

Finally, we derive the following proposition.
\begin{prop}\label{prop:regular-normalized}
For any $d$-regular connected graph $G$ with $d\ge 3$, the following conditions are equivalent:
\begin{itemize}
\item $\mathrm{gap}(G)=\frac{\sqrt{d-1}}{d}$
\item $\mathrm{gap}(G)>\frac{\sqrt{d-2}}{d}$
\item $G$ is the incidence graph of a finite projective plane of order $d-1$ 
\end{itemize}    
\end{prop}

The proof of Theorem \ref{thm:d-graph-projective} is then completed.

\subsection{Proof of Theorem \ref{th:main}}

For the case of connected $d$-regular graph, Theorem \ref{th:main} follows from Proposition \ref{prop:regular-normalized} and the fact that $\min\limits_{\lambda\in \sigma(\mathbf{A})}|\lambda|=d\cdot \mathrm{gap}(G)$. 

The rest of the proof is a detailed analysis for non-regular graphs.  It suffices to prove that for any non-regular graph $G\in\mathcal{G}_{\le d}$, 
\begin{equation}\label{eq:nonregular-d-2}
\min_{\lambda\in \sigma(\mathbf{A})}|\lambda|\le \sqrt{d-2}.
\end{equation}
Our approach is based on spectral interactions between $G$ and its neighborhood graph $\phi(G)$.
\begin{prop}\label{prop:minimum-degree-adjacency}
For any graph $G$, 
\begin{equation}\label{eq:adjacency-min-degree}
\min_{\lambda\in \sigma(\mathbf{A})}|\lambda|\le \sqrt{\min\limits_{u\in V}\deg u}    
\end{equation}
with equality if and only if $G$ has a component that is isomorphic to $K_2$ or $K_1$.
\end{prop}
\begin{proof}
    
Similar to Lemma \ref{coro:1}, we have 
\begin{equation}\label{eq:adjacency-min-ng}
\big(\min_{\lambda\in \sigma(\mathbf{A})}|\lambda|\big)^2=\min_{f:V\rightarrow{\mathbb{R}},f\ne\mathbf{0}} \mathcal{R}_{\mathbf{A}^2}(f)
\end{equation}
where $$\mathcal{R}_{\mathbf{A}^2}(f):=\frac{\sum_{u,v\in V}{ |\mathcal{N}(u)\cap\mathcal{N}(v)|f(u)f(v)}}{\sum_{w\in V}f(w)^2}.$$
Take $f_u:V\to\mathbb{R}$ defined as $f_u(u)=1$ and $f_u(v)=0$ whenever $v\ne u$. Then, it follows from \eqref{eq:adjacency-min-ng} that for any $u\in V$, $
\big(\min_{\lambda\in \sigma(\mathbf{A})}|\lambda|\big)^2\le \mathcal{R}_{\mathbf{A}^2}(f_u)=\deg(u)$ and thus we obtain \eqref{eq:adjacency-min-degree}.

We now focus on the equality case. Suppose that \eqref{eq:adjacency-min-degree} holds with equality. 
Then, there exists a vertex with minimum degree $\deg u=\min_{v\in V}\deg v$, and $f_u$ is an eigenvector of $A^2$, i.e., $A^2f_u=\deg (u)\, f_u$ but this implies that $u$ has no neighbor in $\phi(G)$, meaning that the component containing $u$ is $K_2$ or the singleton $K_1$. 
\end{proof}

We remark here that Proposition \ref{prop:minimum-degree-adjacency} is a generalization of Theorem 5 in \cite{Mohar15}.

\vspace{0.2cm}

We are now ready to prove \eqref{eq:nonregular-d-2}. 
Suppose the contrary, that there exists a non-regular graph 
$G\in\mathcal{G}_{\le d}$ satisfying $\min\limits_{\lambda\in \sigma(\mathbf{A})}|\lambda|>\sqrt{d-2}$. Proposition \ref{prop:minimum-degree-adjacency} immediately implies that $\deg u\in \{d-1,d\}$ for any $u\in V$.

Similar to the inequality \eqref{eq:repeat-key-opti-d}, it is easy to see
\begin{equation}\label{eq:2dN(u)N(v)}
2d\ge\deg u+\deg v> 2d+2|\mathcal{N}(u)\cap \mathcal{N}(v)|-4
\end{equation}
whenever $u\ne v$. This implies that for any $w\in V$, there is at most one $v\in \mathcal{N}(w)$ with $\deg v=d-1$. 

The inequality \eqref{eq:2dN(u)N(v)} also yields $|\mathcal{N}(u)\cap \mathcal{N}(v)|\le 1$ for any distinct $u$ and $v$, and thus, 
$G$ is 4-cycle free.


Argument 1: If $T$ is an induced subtree of $\phi(G)$, then there is at most one vertex $v\in V(T)$ with $\deg v=d-1$. 

Proof of Argument 1: 
Since every tree is bipartite, there exists a test function $f_T:V\to\{-1,0,1\}$ 
such that $f_T(v)f_T(u)=-1$ whenever $u$ and $v$ are adjacent in $T$, and $f_T(w)=0$ for any $w\not\in V(T)$. Then 
$$\mathcal{R}_{\mathbf{A}^2}(f_T)=\frac{\sum_{v\in V(T)}\deg v-2(|V(T)|-1)}{|V(T)|}>d-2$$
which implies $\sum\limits_{v\in V(T)}\deg v+2>d|V(T)|$. Since $\deg v\in\{d-1,d\}$, we have $$|\{v\in V(T):\deg v=d-1\}|\le 1$$ which completes the proof of Argument 1.

\vspace{0.1cm}
Now, if there exist two vertices $u$ and $v$ in a connected component of $\phi(G)$ with $\deg u=\deg v=d-1$, then the shortest path $T$ in $\phi(G)$ connecting $u$ and $v$ is an induced subtree, which contradicts Argument 1. Therefore, every connected component of $\phi(G)$ has at most one vertex $v$ with $\deg v=d-1$. 

The remainder is a slight modification to the proof of Theorem 3 in \cite{Mohar15}. 
Since $G$ is 4-cycle free, for a vertex $v$ with $\deg v=d-1$, there holds $\deg_{\phi(G)}(v)=\sum_{u\in \mathcal{N}(v)}\deg u-\deg(v)\in\{(d-1)^2-1,\,(d-1)^2\}$, and similarly, $\deg_{\phi(G)}(u)=d(d-1)$ or $d(d-1)-1$ for any other vertex $u$ in the component of $\phi(G)$ containing $v$.

If $d\ge 4$, then $\phi(G)$ has at least 9 vertices, and thus by \cite[Theorem 2.1]{Doob} or \cite[Theorem 2.3.20]{Cvetkovicbook}, $\phi(G)$ must be the line graph of a tree, or the line graph of a (multi-)graph formed by adding an edge to a tree. Suppose $\phi(G)=\mathrm{Line}(P)$, where $P$ satisfies $|V(P)|\ge |E(P)|$. Since $P$ is not a cycle (otherwise $\phi(G)$ is a cycle which contradicts $d\ge 4$), there exists a vertex $\alpha\in P$ with $\deg_P \alpha=1$. Let $\beta$ be the unique neighbor of $\alpha$ in $P$, and let $x$ be the vertex in $\phi(G)$ corresponding to the edge $\alpha\beta\in E(P)$. Then we have $\deg_{\phi(G)}(x)\ge (d-1)^2-1$ implying that $\deg_P(\beta)\ge (d-1)^2$. Let $k$ be the number of vertices of degree 1 in $P$, then $k+\deg_P(\beta)+2(|V(P)|-1-k)\le \sum_{\gamma\in V(P)}\deg_P(\gamma)= 2|E(P)|\le 2|V(P)|$ and consequently, $k\ge \deg_P(\beta)-2\ge (d-1)^2- 2\ge d+3$ by $d\ge 4$. 

\begin{prop}\label{pro:non-bipartite=phiG-connected}
A connected graph $G$ is non-bipartite iff $\phi(G)$ is connected. 
\end{prop}

For readers' convenience, we provide a proof of Proposition \ref{pro:non-bipartite=phiG-connected} in the appendix.


Thanks to Proposition \ref{pro:non-bipartite=phiG-connected}, either $\phi(G)$ is connected itself, or $\phi(G)$ has two connected components. In either case, there are at most $1+d-1=d$ vertices in each component of $\phi(G)$ with $\phi(G)$-degree less than $d(d-1)$. Since $k>d$, we can take $x\in V(P)$ with $\deg_P(x)=1$ such that the unique neighbor $y$ of $x$ in $P$ satisfies $\deg_P(y)= d(d-1)+1$. This implies that each component of $\phi(G)$ has a clique of order $d(d-1)+1$, and thus each component of $\phi(G)$ is the complete graph of order $d(d-1)+1$. Similar to the proof of Lemma \ref{lemma:phi(G)}, we have confirmed the case $d\ge 4$ of Theorem \ref{th:main}. 

The remaining case $d=3$ directly follows from a very recent progress on subcubic graphs. Precisely, the main theorem in \cite{Acharya} states that $R(G)\le 1$ for any chemical graph $G$ except for the Heawood graph. Since $\min_{\lambda\in \sigma(\mathbf{A})}|\lambda|\le R(G)$, the case of $d=3$ is a direct consequence of \cite[Theorem 1.1]{Acharya}.

\begin{remark}\label{remark:th1-equivalent}
We point out 
that Theorem \ref{th:main} is essentially equivalent to \cite[Theorem 4]{Mohar15}. 
Since the spectrum of a bipartite graph is origin-symmetric, the absolute value of the mean eigenvalue is equal to $\min_{\lambda\in \sigma(\mathbf{A})}|\lambda|$. This means that Theorem \ref{th:main} formally generalizes \cite[Theorem 4]{Mohar15}. 

On the other hand, \cite[Theorem 4]{Mohar15} also implies Theorem \ref{th:main}. In fact, given a non-bipartite graph $G$, take its bipartite double cover, say $\widetilde{G}$. If $G$ has no eigenvalues in some origin-symmetric interval, then neither does the bipartite graph $\widetilde{G}$. We can then use \cite[Theorem 4]{Mohar15} to derive Theorem \ref{th:main}.
\end{remark}

\subsection{Proof of Proposition \ref{thm:girth7}}
Since $G$ is a $d$-regular graph of girth at least 7, $G$ must be 4-cycle free. According to Lemma \ref{lemma:phi}, the neighborhood graph $\phi(G)$ 
is $(d^2-d)$-regular. 
We then have for any $f$,
$$\min_{\lambda\in \sigma(\mathbf{A})}|\lambda|^2\le\mathcal{R}_{\mathbf{A}^2}(f)= d+\frac{2\sum_{\{u,v\}\in E(\phi(G))}{ f(u)f(v)}}{\sum_{w\in V}f(w)^2}.$$
Fixed a vertex $u\in V$, let $v_1,\cdots,v_d$  denote the neighbors of $u$ in $G$. Suppose $\mathcal{N}(v_i)=\{u,w_{i,1},\cdots,w_{i,d-1}\}$. Since $G$ is 4-cycle free, the vertices $w_{i,j}$ (for $i=1,\cdots,d$ and $j=1,\cdots,d-1$) are pairwise distinct. Moreover, for any $i$, $\{u,w_{i,1},\cdots,w_{i,d-1}\}$ forms a clique in the neighborhood graph $\phi(G)$. 
Since the girth of $G$ is at least 7, $w_{i,j}$ is not adjacent to $w_{i',j'}$ whenever $i\ne i'$, otherwise there exists a cycle of order 6 in $G$. 
Therefore, the subgraph of $\phi(G)$ induced by $\{u,w_{i,j}:i\in\{1,\cdots,d\},~j\in\{1,\cdots,d-1\}\}$ is the windmill graph 
obtained by taking $d$ copies of the complete graph $K_d$ with a vertex in common. 
Using the test function $\widetilde{f}:V\to\mathbb{R}$ defined as 
$$\widetilde{f}(w)=\begin{cases}
\frac12\big(d-2+\sqrt{d^2+4(d-1)^2}\big),&\text{ if }w=u\\
-1,&\text{ if }w=w_{i,j}\text{ for some }(i,j),
\\0,&\text{ otherwise}
\end{cases}$$ 
we derive 
\begin{align*}
\min_{\lambda\in \sigma(\mathbf{A})}|\lambda|^2&\le \mathcal{R}_{\mathbf{A}^2}(\widetilde{f})= d+\frac{2\sum\limits_{i=1}^{d}\Big(\sum\limits_{j=1}^{d-1}{\widetilde{f}(u)\widetilde{f}(w_{i,j})}+\sum\limits_{\{j,j'\}\subset\{1,\cdots,d-1\}}\widetilde{f}(w_{i,j})\widetilde{f}(w_{i,j'})\Big)}{\widetilde{f}(u)^2+\sum_{i=1}^{d^2-d}\widetilde{f}(v_i)^2}
\\&=d-\frac{\sqrt{d^2+4(d-1)^2}-(d-2)}{2}=d-c_d
\end{align*}
where $c_d$ is the positive root of the quadratic equation $t^2+(d-2)t-d(d-1)=0$. 
This completes the proof. 

\section{
Discussions and open problems}\label{section:discuss}

\begin{itemize}

\item
As stated in Remark \ref{remark:th1-equivalent}, Theorem \ref{th:main} is a spectral gap reformulation of the work on the HL-index of bipartite graphs by Mohar and Tayfeh-Rezaie \cite{Mohar15}. 
 There is an open problem asking whether the HL-index of every graph in $\mathcal{G}_{\le d}$ is bounded by $
 \sqrt{d-1}$ 
 when $d-1$ is a prime power \cite{Mohar}. In some sense, Theorem \ref{th:main} establishes a weak version of this conjecture, and has strengthened the belief in this conjecture. 
 
For $d$-regular graphs with girth at least 7, the constant $c_d$ appearing in Proposition \ref{thm:girth7} satisfies $c_3=2$ and $\lim_{d\to+\infty}c_d/d=(\sqrt{5}-1)/2\approx 0.618$. Thus, if we work on regular graphs with large degrees and girths, the upper bound $\sqrt{d-1}$ for the adjacency spectral gap from 0 is significantly improved to $\sqrt{d-c_d}$. 


\item
We note that the extremal graphs for adjacency spectral gap from 0 (Theorem \ref{th:main}) and that for normalized Laplacian spectral gap from 1 (Theorem \ref{th:main2} and Theorem \ref{thm:d-graph-projective}) coincide. Precisely, the following equality characterize the family of the incidence graphs of finite projective planes of order $d-1$:
$$\left\{G\in\mathcal{G}_{\le d}\left|\min_{\lambda\in \sigma(\mathbf{A})}|\lambda|= \sqrt{d-1}\right.\right\}=\left\{G\in\mathcal{G}_{\ge d}\left|\min\limits_{\lambda\in\sigma(\mathbf{L})}|\lambda-1|= \frac{\sqrt{d-1}}{d}\right.\right\}$$

It should be noted that the proof for the normalized Laplacian case is much more difficult as the interlacing property no longer holds. 

\item 
By using the normalized adjacency matrix $\mathbf{D}^{-1}\mathbf{A}$ instead of $\mathbf{L}$, we can better appreciate the marvel of this contrast: 
\begin{itemize}
    \item 
For any connected graph $G$ with {maximum} degree ${\large\le }~ d$, the spectral gap from 0 with respect to $\mathbf{A}$ is at most $\sqrt{d-1}$, with equality if and only if $G$ is the incidence graph of a projective plane of order $d-1$. 
\item
For any connected graph $G$ with {\textbf{minimum}} degree ${\large\ge }~ d$, the spectral gap from 0 with respect to $\mathbf{D}^{-1}\mathbf{A}$ is at most $\sqrt{d-1}/d$, with equality if and only if $G$ is the incidence graph of a projective plane of order $d-1$. 
\end{itemize}

In other words, Theorems \ref{th:main} and \ref{th:main2} indeed show that for $d\ge 3$,
\begin{align*}
&\left\{G\in\mathcal{G}_{\le d}\left|\min_{\lambda\in \sigma(\mathbf{A})}|\lambda|= \sqrt{d-1}\right.\right\}
=\left\{G\in\mathcal{G}_{\ge d}\left|\min\limits_{\lambda\in\sigma(\mathbf{D}^{-1}\mathbf{A})}|\lambda|= \frac{\sqrt{d-1}}{d}\right.\right\}
\\=~&~\big\{\text{incidence graphs of }PG(d-1,2)\big\}
\end{align*}
where $\sigma(\mathbf{D}^{-1}\mathbf{A})$ denotes the spectrum of $\mathbf{D}^{-1}\mathbf{A}$, and $PG(d-1,2)$ denotes a projective plane of order $d-1$. 

While Theorem \ref{thm:d=2} shows that when $d=2$, the situation is entirely different:
\begin{itemize}
    \item
For any connected graph $G$ with maximum degree $\le 2$, the spectral gap from 0 with respect to $\mathbf{A}$ is at most $1$, with equality if and only if $G$ is a triangle or a hexagon. 
\item
For any connected graph $G$ with {\textbf{minimum}} degree $\ge 2$, the spectral gap from 0 with respect to $\mathbf{D}^{-1}\mathbf{A}$ is at most $1/2$, with equality if and only if $G$ is a friendship graph or a book graph. 
\end{itemize}
In short, this means
$$\left\{G\in\mathcal{G}_{\le 2}\left|\min_{\lambda\in \sigma(\mathbf{A})}|\lambda|= 1\right.\right\}\subsetneqq\left\{G\in\mathcal{G}_{\ge 2}\left|\min\limits_{\lambda\in\sigma(\mathbf{D}^{-1}\mathbf{A})}|\lambda|= \frac{1}{2}\right.\right\}$$
\item
Following Koll\'ar and Sarnark \cite{Kollar20}, a \emph{gap interval} for the normalized Laplacian spectra of graphs in $\mathcal{G}$ is an open interval such that there are infinitely many graphs in $\mathcal{G}$ whose normalized Laplacian spectrum does not intersect the interval. 
Theorem \ref{thm:d=2} implies that $(\frac12,\frac32)$ is a maximal gap interval for the normalized Laplacian spectra of graphs in $\mathcal{G}_{\ge 2}$. On the other hand, Theorem \ref{thm:d-graph-projective} implies that $(1-\frac{\sqrt{d-1}}{d},1+\frac{\sqrt{d-1}}{d})$ is \emph{not} a gap interval for the normalized Laplacian spectra of graphs in $\mathcal{G}_{\ge d}$ when $d\ge 3$. 

\item 
By Theorem \ref{th:main}, there is no graph $G\in\mathcal{G}_{\le d}$ with $\min_{\lambda\in \sigma(\mathbf{A})}|\lambda|\in (\sqrt{d-2},\sqrt{d-1})$. 
One may expect to see a similar statement on the normalized Laplacian, that is, there is no graph 
$G\in\mathcal{G}_{\ge d}$ with $\min_{\lambda\in \sigma(\mathbf{L})}|\lambda-1|\in (\frac{\sqrt{d-2}}{d},\frac{\sqrt{d-1}}{d})$. 
However, 
it seems that this statement is false when $d=3$. The reason presented below is inspired by \cite{Mohar15}.

It is known that the incidence graph of a biplane $(\mathrm{v},d,2)$ has degree $d$ and smallest adjacency eigenvalue $\sqrt{d-2}$ in absolute value. Since the biplanes $(\mathrm{v},d,2)$ exist when $d\in \{ 2, 3, 4, 5, 6, 9, 11, 13\}$ (see \cite{Hall}), there exist $3$-regular graphs with $\mathrm{gap}(G)=\frac{1}{3}$, $4$-regular graphs with $\mathrm{gap}(G)=\frac{\sqrt{2}}{4}$, $5$-regular graphs with $\mathrm{gap}(G)=\frac{\sqrt{3}}{5}$, and $6$-regular graphs with $\mathrm{gap}(G)=\frac13$. Note that $\frac13<\frac{\sqrt{3}}{5}<\frac{\sqrt{2}}{4}<\frac{\sqrt{2}}{3}$. So, there exists $G\in\mathcal{G}_{\ge 3}$ with $\mathrm{gap}(G)=\frac{\sqrt{2}}{4}\in (\frac13,\frac{\sqrt{2}}{3})$. 

\end{itemize}


We present some questions related to the main theorems in this paper.


\begin{ques}\label{que:1}
Suppose that $d\ge 3$ and $d-1$ is not the order of any finite projective plane. Determine the exact values of $\mathbf{gap}(\mathcal{G}_{=d})$ and $\mathbf{gap}(\mathcal{G}_{\ge d})$, respectively. 
\end{ques}

We conjecture that the extremal graphs for $\mathbf{gap}(\mathcal{G}_{=d})$ and $\mathbf{gap}(\mathcal{G}_{\ge d})$ are incidence graphs of the $(\mathrm{v},d,\lambda)$ designs.

\begin{ques}\label{que:2}
Are the finite projective planes the extremal graphs of the spectral gap from the average of eigenvalues with respect to the graph \emph{unnormalized} Laplacian?
\end{ques}

If the answer to Question \ref{que:2} is affirmative, it would be very interesting to explore \emph{the spectral gap from average}, and study which graph matrix satisfies this property.



 
\section*{Acknowledgments} The authors are very grateful to Zilin Jiang for pointing out a simple approach to Theorem \ref{th:main} by employing the results of Mohar and Tayfeh-Rezaie. 
This work is supported by grants from the National Natural Science Foundation of China (No. 12401443).

\section*{Appendix}

\begin{proof}[Proof of Lemma \ref{lemma:phi}]
Fixed a vertex $x\in V$, let $\mathcal{N}(x)=\{y_1,\cdots,y_d\}$ and $$\mathcal{N}(y_k)=\{x,z_{k,1},\cdots,z_{k,d-1}\}$$ where $k=1,\cdots,d$. We claim $z_{k,l}\ne z_{k^*,l^*}$ whenever $(k,l)\ne(k^*,l^*)$. In fact, we will prove that $z_{k,l}= z_{k^*,l^*}$ implies $(k,l)=(k^*,l^*)$. 

Suppose $z_{k,l}= z_{k^*,l^*}$. 
If $k\ne k^*$, then $|\mathcal{N}(y_k)\cap\mathcal{N}(y_{k^*})|\ge2$ as $x$ and $z_{k,l}= z_{k^*,l^*}$ are two distinct vertices in $\mathcal{N}(y_k)\cap\mathcal{N}(y_{k^*})$, but this contradicts the 4-cycle free condition. 
Hence, we have $k=k^*$. Since $z_{k,l},z_{k,l^*}\in \mathcal{N}(y_k)$, $z_{k,l}= z_{k,l^*}$ implies $l=l^*$, meaning that $(k,l)=(k^*,l^*)$. 

Next, we prove that $\mathcal{N}_{\phi(G)}(x)=\{z_{k,l}\mid1\le k\le d,1\le l\le d-1\}$, and thus $|\mathcal{N}_{\phi(G)}(x)|=d^2-d$. 

On the one hand, if $w\in\mathcal{N}_{\phi(G)}(x)$, then there exists $y\in\mathcal{N}(x)\cap\mathcal{N}(w)$, and hence there is $k$ such that $y=y_k$, and subsequently, there is $l$ such that $w=z_{k,l}$. Therefore, $$\mathcal{N}_{\phi(G)}(x)\subseteq\{z_{k,l}\mid1\le k\le d,1\le l\le d-1\}.$$ 
On the other hand, for any $z_{k,l}\in\{z_{k,l}\mid1\le k\le d,1\le l\le d-1\}$, we have $y_k\in \mathcal{N}(x)\cap\mathcal{N}(z_{k,l})$ and therefore, $z_{k,l}\in \mathcal{N}_{\phi(G)}(x)$. The proof is then completed.
\end{proof}

\begin{proof}[Proof of Proposition \ref{pro:non-bipartite=phiG-connected}]
If $G$ is bipartite, then clearly $\phi(G)$ has exactly two connected components. We refer to \cite[Lemma 5.3]{Bauer13} for details. 

Suppose that $G$ is non-bipartite. 
Then there exists an odd cycle in $G$.  We shall prove that for any two distinct vertices $u$ and $v$ in $G$, there exists a path of even length connecting $u$ and $v$. 
Let $C$ be an odd cycle in $G$, and fix a vertex $w\in C$. Consider a shortest path from $u$ to $w$, and a shortest path from $w$ to $v$. If the path $u\sim w\sim v$ made up of the two shortest paths is of even length, then the proof is done. Otherwise, the length of the path $u\sim w\sim v$ made up of the two shortest paths is odd, and then the odd-length cycle $C$ can be further merged to create an even-length path $u\sim w\mathop{\sim}\limits^C w\sim v$ connecting $u$ and $v$. 

Note that an even-length path connecting $u$ and $v$ in $G$ generates a path connecting $u$ and $v$ in $\phi(G)$. Therefore, $\phi(G)$ is connected. 
\end{proof}

\end{document}